\documentclass[11pt]{amsart}
\usepackage{graphicx} 

\usepackage[utf8]{inputenc}
\usepackage{amsmath}
\usepackage{xcolor}
\usepackage{amssymb}
\usepackage{amsmath,stackengine}
\usepackage{amsfonts}
\usepackage{amstext}
\usepackage{amsthm}
\usepackage[english]{babel}
\usepackage[autostyle]{csquotes}
\usepackage{hyperref}
\usepackage[margin=1.28in]{geometry}

\usepackage[pagewise]{lineno}

\newtheorem{theorem}{Theorem}[section]
\newtheorem{proposition}[theorem]{Proposition}

\newtheorem{corollary}[theorem]{Corollary}
\newtheorem{definition}[theorem]{Definition}
\newtheorem*{remark}{Remark}

\newcommand{\rn}{\mathbb{R}^{n}}
\newcommand{\R}{\mathbb{R}}

\newcommand{\Sl}{S_{\ell}}
\newcommand{\ul}{u_{\ell}}
\newcommand{\ui}{u_{\infty}}
\newcommand{\1}{\omega_{1}}
\newcommand{\2}{\omega_{2}}

 \numberwithin{equation}{section}

\hypersetup{
    colorlinks=true,
    linkcolor= blue,
    citecolor=blue,
    pdfpagemode=FullScreen,
    }
{%
\setcounter{enumi}{0}

\begin{enumerate}}%
{\end{enumerate} }

\makeatletter
\newcommand{\myitem}[1]{%
\medskip \item[#1]\protected@edef\@currentlabel{#1}%
}

\title[SOLUTIONS OF QUASILINEAR p-LAPLACE EQUATION ON INFINITE CYLINDERS]{ASYMPTOTICS OF LARGE SOLUTIONS OF p-LAPLACE EQUATION ON CYLINDERS BECOMING UNBOUNDED}

\author[N. N. Dattatreya]{N. N. Dattatreya}
\address{\parbox{.8\linewidth}
{{\textbf{N. N. Dattatreya}}\medskip \\
Indian Institute of Technology - Kanpur, India \medskip}}
\curraddr{}
\email{dattatreya21@iitk.ac.in}
\date{}

\smallskip


\begin{document}
\begin{abstract}
In this article, we study the asymptotic behavior of large solutions for a quasi-linear equation involving the p-Laplacian, defined on a sequence of finite cylindrical domains converging to an infinite cylinder. We demonstrate that the sequence of solutions converges locally, in the Sobolev norm, to a solution of the corresponding cross-sectional problem. Moreover, we establish a convergence rate. As part of our analysis, we extend existing convergence results for the case $p\geq 2$, which previously lacked explicit convergence rates, to the range $1<p<2$. We additionally address solutions with finite Dirichlet boundary data within a unified framework and exhibit that this rate of convergence is independent of the boundary data.
\end{abstract}
\maketitle
\section{Introduction}
The problems motivated by hadron colliders and other parts of physics entail studying the asymptotic behaviour of the sequence of solutions to PDEs defined on finite cylinders becoming unbounded in a particular direction or directions. The goal is to capture the essence of the solution or solutions to a similar PDE on the corresponding infinite cylinder. Such types of asymptotics are studied for a broad class of equations, including linear, semilinear and quasilinear. While much of the study focuses on the equation involving a forcing term, the nonlinearity case is scarcely addressed. In this context, the present article investigates similar asymptotic results for quasilinear problems involving p-Laplacian, wich is defined as follows:
Let 
\begin{enumerate}
    \item $\1$ be a convex domain in $\R^{m}$ such that $0\in \1$.
    \item $\2$ be an open bounded domain in $\R^{n-m}$.
    \item  $\Sl=\Sl(\1,\2):=\ell\1\times\2$ and $S=S(\2):=\R^{m}\times \2$
    \item $f:[0,\infty)\to [0,\infty)$ be a non-decreasing continuous function such that $f(0)=0$.
    \item $g: S \to \R$ be a $W_{loc}^{1,p}(S)$ function such that $g(x)=g(X_{2})$ and $F(g(x))\in L_{loc}^{1}( S)$, where $x=(X_{1},X_{2})\in 
\partial S$, 
$X_{2}\in \partial \2$ and
    \begin{equation}\label{F}
        F(x)=\int_{0}^{x}f(s)\ ds.
    \end{equation}
\end{enumerate}
We consider
\begin{equation}\label{eqnonsl}
    \text{div}\left(|\nabla \ul|^{p-2}\nabla  \ul\right)=f(\ul) \quad \text{in}\ \Sl,
\end{equation}
with either of the following boundary conditions:
\begin{equation}\label{bdryonsl_one}
    \ul(x)=g(x) \quad \text{on}\ \partial \Sl,
\end{equation} 
\begin{equation}\label{bdryonsl_two}
    \ul(x)\to +\infty \quad \text{as}\ \text{dist}(x,\partial\Sl)\to 0,
\end{equation}
And the corresponding cross-sectional problem 
\begin{equation}\label{eqnoncross}
        \text{div}\left(|\nabla \ui|^{p-2}\nabla  \ui\right)=f(\ui) \quad \text{in}\ \2,
\end{equation}
where, the gradient is in $\R^{n-m}$. With the notation $x=(X_{1},X_{2})\in \Sl$, where $X_{1}\in \ell\1$ and $X_{2}\in \2$, we consider either of the following boundary conditions:
\begin{equation}\label{bdryoncross_one}
        \ui(X_{2})=g(X_{2}) \quad \text{on}\ \partial\2
\end{equation}
\begin{equation}\label{bdryoncross_two}
    \ui(X_{2})\to+\infty \quad \text{as}\ \text{dist}(X_{2},\partial\2)\to 0.
\end{equation}
Solutions that blow up at the boundary, as in \eqref{bdryonsl_two} and \eqref{bdryoncross_two} are called \textit{Boundary blow-up solutions} or \textit{Large solutions}, we refer to \cite{materothisis, IndroDatta2024} and the references therein for more details.
The necessary and sufficient condition on the nonlinearity $f$ that ensures the existence of a large solution is called the \textit{Keller-Osserman} condition: 
\begin{enumerate}
    \myitem{$\textbf{(A1)}$}\label{a1}
    \begin{equation*}
         \Psi_{p}(r):=\biggl(1-\frac{1}{p}\biggl)^{\frac{1}{p}}\int_{r}^{\infty}\frac{dx}{F(s)^{\frac{1}{p}}}<\infty \quad \text{for all}\ r>0,
    \end{equation*}
   where, $F$ is defined in \eqref{F}.
\end{enumerate}
The Keller-Osserman condition first appeared in 1957 in the independent works of Keller \cite{Keller1957} and Osserman \cite{Osserman1957}, which were to aid Electrodynamics and Geometry, respectively. For the existence of finite solutions, as in \eqref{bdryonsl_one} and \eqref{bdryoncross_one}, assumption \eqref{a1} is not required; the existence and uniqueness of a finite, locally $C^{1,\alpha}$ weak solution are addressed in \cite{diaz1985} and $C^{1,\alpha}$ regularity in \cite{DiBenedetto1983}. The existence of solutions can also be obtained by \textit{monotonicity methods} \cite{Evansbook2010}. For more regarding p-Laplacian we refer to \cite{lindqvist} and \cite{Ollimartiobook2006}. A priori estimates for positive solutions can be found in \cite{RosaLucio2018} and references therein. We refer to \cite{materopaper, diazG1993, materothisis} for the existence of large solutions. The uniqueness of a large solution requires an additional assumption on $f$, we refer to \cite{materothisis}, and the condition is as follows
\begin{enumerate}
\myitem{$\textbf{(A2)}$}\label{a2}
\begin{equation*}
\displaystyle \liminf_{t\to\infty}\frac{\Psi_{p}( t)}{\Psi_{p}(t)}>1\quad \text{for all}\ \in (0,1).
\end{equation*}
\end{enumerate}
For more on the uniqueness we refer to \cite{Zongming2007}. 
With this setup, we derive the following theorems, where, the first one addresses the asymptotics of solutions with finite boundary data
\begin{theorem}\label{thrm1}
    We assume \eqref{a1}. Let $\ul$ solve \eqref{eqnonsl}, \eqref{bdryonsl_one}, and $\ui$ solve \eqref{eqnoncross}, \eqref{bdryoncross_one}, then there exists a constant $C$ independent of $\ell$ such that 
    \begin{equation*}
        ||\nabla(\ul-\ui)||_{L^{p}(\Tilde{S})}\leq \frac{C}{\ell^{1/p}},
    \end{equation*}
    for any $\Tilde{S}\subset\subset S$ and $\ell$ sufficiently large.
    \end{theorem}
And the second one addresses large solutions,
\begin{theorem}\label{thrm2}
    We assume \eqref{a1}. Let $\ul$ be the solution to \eqref{eqnonsl} and \eqref{bdryonsl_two}, then
    \begin{enumerate}
        \item There exist $u\in W_{loc}^{1,p}(S)$ such that $\ul \rightharpoonup u$ in $W^{1,p}_{loc}(S)$ and for every $X_{1}\in \1$, $v\left(X_{2}\right):=u\left(X_{1},X_{2}\right)$ solves \eqref{eqnoncross} with \eqref{bdryoncross_two}.
        \item With the assumption \eqref{a2}, $u(X_{1},X_{2})=\ui(X_{2})$ for every $X_{1}\in \1$.  
        \item There exists a positive constant $C$ independent of $\ell$  such that
    \begin{equation*}
        ||\nabla(\ul-u)||_{L^{p}(\Tilde{S})}\leq \frac{C}{\ell^{1/p}},
    \end{equation*}
    for any $\Tilde{S}\subset\subset S$ and $\ell$ sufficiently large.
     \end{enumerate}
\end{theorem}
The theorems are stated separately for clarity, but they will be proved simultaneously. The assumption \eqref{a1} in Theorem \ref{thrm1} is a technical requirement; to drop it, one would need to replace Proposition \ref{prop_weak_local_bound} with an alternative bound that does not rely on the existence of a large solution in one dimension. Our approach unifies arguments for large solutions and finite solutions. Also, we aim to obtain a convergence rate for large solutions as an addendum to the work \cite{IndroDatta2024} and extend the results to solutions with finite boundary values.

Cylindrical domains naturally arise in various applications, such as measuring pressure in water or gas pipelines, modelling gas flow in pipes \cite{Gasequation}, and theoretically in the study of cylindrical gravitational waves, which originate from cylindrical sources and propagate along the axial direction \cite{BuzzMatKir2023}, They also appear in the context of porous media flow \cite{BenHua2015}. \\
The asymptotic behaviour of solutions in the above-mentioned settings has been extensively studied for various operators having forcing terms with different boundary conditions and eigenvalue problems. And we refer to some of the works \cite{ChipotRougirel2002, ChipoXie2004, ChioptRoyShafi2013, PJana2024}. A study of said asymptotics of solutions is carried out in the case of a uniformly elliptic operator with Dirichlet boundary data, and the convergence of eigenvectors is studied in \cite{ChipotRougirel2008}. Semilinear equations are studied in \cite{ChipoPino2017}.
\smallskip

This work extends the results discussed in \cite{IndroDatta2024} in three aspects: While the cited article focuses on infinite boundary data, our focus is on both finite and infinite boundary data. The result in \cite{IndroDatta2024} is established for $2\leq p< \infty$. In the current article, we extend them to $1<p<2$. Additionally, this article provides a rate for convergence. Although the proof is a simple bookkeeping, the rate and the extension to $1<p<2$ case together justify the article. 
\smallskip

In section 2, we give definitions and preliminary results. In section 3, we prove Theorem \ref{thrm1}, and Theorem \ref{thrm2} with an additional assumption \eqref{a2}. Section 4 discusses the loss of uniqueness; here, we also give the proof of the case $1<p<2$.

\section{Preliminaries}
We began by defining the notation of a \textit{solution}, \textit{subsolutions} and \textit{supersolutions}, followed by a statement of \textit{Comparison Principle}. Subsequently, we derive bounds for $\{\ul\}_{\ell}$ and $\left\{||\nabla \ul||_{L^{p}}\right\}_{\ell}$. However, before proceeding, we must examine a one-dimensional problem. For $r>0$, consider
\begin{equation*}
\begin{cases}
    \left(\left|\phi'\right|^{p-2}\phi'\right)'=f\left(\phi\right) \quad &\text{in}\ (-r,r)\\
    \phi(t)\to +\infty \quad &\text{as}\ t\to \pm r
    \end{cases}
\end{equation*}
one can implicitly solve the above equation uniquely, see \cite{IndroDatta2024}. However, we are interested in the convexity of $\phi$. By the product rule, we have 
\begin{equation*}
    (p-1)|\phi'|^{p-2}\phi''=f(\phi).
\end{equation*}
Now, the convexity of $\phi$ follows the sign of $f$. Also, since $\phi(-t)$ solves the above equation, by uniqueness, $\phi$ must be symmetric about the origin. With which we infer that $\phi$ is increasing in $(0,r)$.
\begin{definition}\label{dfn_solution}
By a weak solution of div$\left(|\nabla u|^{p-2}\nabla u\right)= f(u)$ in $\Omega$, a domain in $\rn$, we mean function $u\in W_{loc}^{1,p}(\Omega)$ satisfying 
\begin{align*}
    -\int_{\Omega'}|\nabla u|^{p-2}\nabla u\cdot\nabla \phi \ dx=\int_{\Omega'}f(u)\phi \ dx \quad \forall~ \phi\in W_{0}^{1,p}(\Omega'),
\end{align*}
and for every $\Omega'\subset\subset\Omega$.
Moreover, a function $u\in W_{loc}^{1,p}(\Omega)$ is a \textit{weak subsolution} (or weak supersolution) to div$\left(|\nabla u|^{p-2}\nabla u\right)= f(u)$ in $\Omega$, a domain in $\rn$, if
\begin{multline*}
    -\int_{\Omega'}|\nabla u|^{p-2} \nabla u\cdot\nabla \phi \ dx\geq \int_{\Omega'}f(u)\phi \ dx  \ \  \left(\text{or}, \leq  \int_{\Omega'}f(u)\phi \ dx \right) \ \forall~ \phi\in W_{0}^{1,p}(\Omega'),\phi\geq 0,
\end{multline*} 
    and for every $\Omega'\subset \subset \Omega$.
\end{definition}
Where, by $\Omega'\subset\subset\Omega$, we mean that $\overline{\Omega'}$ is a compact subset of $\Omega$.

\begin{proposition}[\cite{diazG1993}]{(Comparison Principle)}\label{prop_comparison}
    If $u,v\in W_{loc}^{1,p}(\Omega)\cap C(\Omega)$ be such that 
    \begin{align*}
        -\text{div}\left(|\nabla u|^{p-2} \nabla u\right)+ f(u)\leq -\text{div}\left(|\nabla v|^{p-2} \nabla v\right)+ f(v)\quad \text{weakly in}\ \Omega,
    \end{align*}
and that, 
\begin{align*}
    \limsup \frac{u(x)}{v(x)}\leq 1\quad \text{as}~dist(x,\partial\Omega)\to 0.
\end{align*} Then $u\leq v$ in $\Omega$. In particular, if $u$ is a weak subsolution and $v$ is a weak supersolution with above boundary behaviour, then the result will hold.
\end{proposition}
We now establish bounds; proofs of which are present in \cite{IndroDatta2024}. The proof of the first precisely follow argument in \cite[Proposition 3.3]{IndroDatta2024}, and is outlined for completeness.

\begin{proposition}\label{prop_weak_local_bound}
    Assume \eqref{a1}, let  $u\in W_{loc}^{1,p}(\Omega)\cap C(\Omega)$ be a weak subsolution of  div$\left(|\nabla u|^{p-2}\nabla u\right)= f(u)$ in a bounded domain $\Omega$. Also let $x_{0}\in \Omega$, for any $R>0$ such that $B_{R}(x_{0})\subset \subset \Omega$, then, we can infer that
    \begin{equation}\label{eqn_local_bound}
        u(x)\leq \phi\left(\frac{R}{2}\right) \quad \text{for all } x\in B_{\frac{R}{2}}(x_{0}),
    \end{equation}
    where $\phi$ solves
     \begin{equation}
        \begin{cases}
        \left(\left|\phi'(t)\right|^{p-2}\phi'(t)\right)'=f\left(\phi(t)\right) \quad &\text{in } (-R,R)\\
            \phi(t)\to\infty \quad &\text{as } t\to\pm R.
        \end{cases}
    \end{equation}
\end{proposition}
 \begin{proof}
     In the proof, we construct a weak large solution $v$ to \eqref{eqnonsl} and \eqref{bdryonsl_two} using the function $\phi$. Then by comparing $v$ with the solution $u$, whether $u$ has finite boundary values or not, we derive the desired bound.

      Define $v(x)=\phi(|x-x_{0}|)\in C(B_{R}(x_{0}))$. For any $\psi\in C_{c}^{\infty}(B_{R}(x_{0}))$, we denote $x=x_{0}+\frac{t}{R}\theta(x)$, where, $t\in (0,R)$ and $\theta(x)\in \partial B_{R}(x_{0})$, it follows by polar coordinate, that
     \begin{equation*}
     \begin{split}
    &\int_{B_{R}(x_{0})}|\nabla v|^{p-2}\nabla v\cdot \nabla\psi \ dx\\
    &=\int_{B_{R}(x_{0})}|\phi'(|x-x_{0})|^{p-2}\phi'(|x-x_{0}|)) \frac{x-x_{0}}{|x-x_{0}|}\cdot \nabla \psi \ dx\\
      &=\int_{0}^{R}\int_{\partial B_{t}(x_{0})}|\phi'(t)|^{p-2}\phi'(t) \theta(x) \cdot \nabla \psi(t,x)  \ dH(x) dt\\
        &=\int_{0}^{R}|\phi'(t)|^{p-2}\phi'(t)~ \int_{\partial B_{R}(x_{0})}\frac{\partial \psi}{\partial t}(t,x) \frac{t^{n-1}}{R^{n}} \ dH(x) dt.\\ 
         \end{split}
     \end{equation*}
    as $\frac{\partial \psi}{\partial t}(t,x)=\theta(x) \cdot \nabla \phi(t,x)$.\\
    Define 
    \begin{equation*}
        \eta(t)=\begin{cases}
            \int_{0}^{t} \int_{\partial B_{R}(x_{0})}\frac{\partial}{\partial s}\left(\psi(s,x) \frac{s^{n-1}}{R^{n}}\right) \ dH(x) ds \quad &\text{for }t> 0\\
            0 \quad &\text{for } t\leq 0, 
        \end{cases}
    \end{equation*}
    which is a $C_{c}(\mathbb{R})$ function, as $\phi\in C_{c}^{\infty}(B_{R}(x_{0}))$. 
Which implies that $v$ is a solution, since $\eta$ acts as a test function against $\phi$.
    $v(x)\to \infty$ as $|x|\to R$, since $u<\infty$ in $B_{R}(x_{0})$, comparison principle implies 
    \begin{equation*}
        u(x)\leq v(x)\quad \text{for all }x\in B_{R}(x_{0}).
    \end{equation*}
    Because $\phi$ is increasing in $(0,R)$, we get \eqref{eqn_local_bound}. 
    \end{proof} 
 Next, we state how the bound obtained above helps us with varying domains.
\begin{corollary}\label{corollary1}
  The family of solution $\{\ul\}_{\ell}$ of \eqref{eqnonsl}, with either \eqref{bdryonsl_one} or \eqref{bdryonsl_two} is locally uniformly bounded for $\ell$ sufficiently large. That is, for any compact subset $K$ of $S$, one can find $\ell(K)>0$  and $M(K)>0$ such that $|\ul(x)|\leq M(K)$ for all $\ell\geq \ell(K)$ and for all $x\in K$.
\end{corollary}
\begin{proposition}\label{prop_grad_lp_bound}
     Let $u\in W_{loc}^{1,p}(\Omega)$ be any weak subsolution to div$\left(Q|\nabla u|)\nabla u\right)=f(u)$ in a bounded domain $\Omega$, then 
     \begin{equation*}
          \int_{K}|\nabla u|^{p} \ dx\leq 2f(\Lambda)\Lambda|K_{1}|+ \frac{2^{2p}(p-1)^{p-1}}{p^{p}}\|\nabla \phi\|^{p}_{L^{p}(K_{1})}\Lambda^{p},
     \end{equation*} 
      for any $K\subset\subset K_{1}\subset\subset \Omega$, where $\Lambda:=\|u\|_{L^{\infty}(K_{1})}$.
\end{proposition}
\begin{proof}
    Let $\phi\in C_{c}^{\infty}(K_{1})$ such that $\phi\leq 1$, $\phi=1$ on $K$, by taking $u\phi^{2}$ as a test function in the weak formulation, increasing-ness of $f$ and using Young's inequality for $c>0$, we get,
    \begin{equation*}
        \begin{split}
            \int_{K_{1}}|\nabla u|^{p}\phi^{2} \ dx &\leq -\int_{K_{1}}f(u)u\phi^{2} \ dx - 2\int_{K_{1}}\phi |\nabla u|^{p-2}u \nabla u\cdot \nabla \phi \ dx\\
            &\leq \int_{K_{1}}f(u)u\phi^{2} \ dx + \int_{K_{1}}\phi|\nabla u|^{p-1}u  |\nabla \phi| \ dx.\\
            &\leq f(\Lambda)\Lambda|K_{1}|+ \frac{2c(p-1)}{p}\int_{K_{1}}\phi|\nabla u|^{p} \ dx \\
            &\quad + \frac{2}{pc^{p-1}}\int_{K_{1}} \phi^{p-1} u^{p}|\nabla \phi|^{p} \ dx.
        \end{split}
    \end{equation*}
            Choosing $c=\frac{p}{4(p-1)}$ we get the result.
\end{proof}

\section{Proofs of Theorems Under Uniqueness Assumptions}
In this section, we analyse the case when the solution is unique, and in the next section we address the scenario of non-uniqueness, which occurs in case of blow-up solution. Therefore in the case of Theorem \ref{thrm2}, we additionally assume \eqref{a2}. 

First, we identify the equation satisfied by the extension of $u_{\infty}$ to $\Sl\subset\R^{n}$, with this, we give the proof of both theorems simultaneously:\\
Let $u(x):=u_{\infty}\left(X_{2}\right)$. For $\phi\in C_{c}^{\infty}(\Sl)$ we have, 
    \begin{equation*}
            \int_{\2} |\nabla_{X_{2}}u_{\infty}|^{p-2}\nabla_{X_{2}}u_{\infty}\cdot \nabla_{X_{2}}\phi\left(X_{1},X_{2}\right)\ dx'=\int_{\2}f(u
            _{\infty}(X_{2}))\phi(X_{1},X_{2})\ dX_{2},
    \end{equation*}
integrating on $\ell\1$, since $\nabla u=(0,\nabla_{X_{2}}u_{\infty})\in\R^{n}$ we get
\begin{equation*}
    \int_{\Sl} |\nabla u(x)|^{p-2}\nabla u(x)\cdot \nabla \phi(x)= \int_{\Sl} f(u(x))\phi(x).
\end{equation*}
\begin{remark}\label{remark}
    As a corollary of Proposition \ref{prop_weak_local_bound} and the preceding calculation, we conclude, for any compact subset $K$ of $S$, by choosing $\ell$ large, that $u$ is bounded on K. Without loss of generality, we denote this bound by $M(K)$.
\end{remark}
 In the following proof, $C_{j},~j=1,2,3$ will be arbitrary constants. We use a vector inequality, see \cite{lindqvist}: For $x,y\in \rn$ there exists a positive constant $c$ such that
     \begin{equation}\label{vector inequlity}
          |x-y|^{p}\leq c ~ \left<|x|^{p-2}x-|y|^{p-2}y \cdot x-y\right>.
     \end{equation}
      
\textbf{Proof of the Theorems under uniqueness assumption }\label{proofs of the theorems}
\begin{proof}
Let $\Tilde{S}\subset\subset S$, and let $K\subset\subset S$ and $\ell>0$ large such that $\Tilde{S}\subset\subset K\subset\subset S_{\ell/2}$. Choose $\phi_{\ell} \in C_{c}^{\infty}(K)$ such that $0\leq \phi_{\ell}\leq 1$, $\phi_{\ell}=1$ on $\Tilde{S}$ and $|\nabla \phi_{\ell}| \leq \frac{c_{1}}{\ell}$ for some $c_{1}>0$. Consider, $\psi_{\ell}:=\phi_{\ell}^{p}(u-\ul)\in W^{1,p}_{0}(\Sl)$ as a test function. By \eqref{vector inequlity}, one can write
\begin{equation}\label{ineq from vector ineq}
\begin{split}
     I&=\int_{\Sl} \phi_{\ell}^{p} |\nabla(\ul-u)|^{p} \ dx \\
     &\leq \int_{\Sl}\phi_{\ell}^{p} \{|\nabla \ul|^{p-2}\nabla \ul-|\nabla u|^{p-2}\nabla u\}\cdot \nabla(\ul-u)\ dx \\
       &=\int_{\Sl} \{|\nabla \ul|^{p-2}\nabla \ul-|\nabla u|^{p-2}\nabla u\}\cdot \{\nabla \psi_{\ell}-p \phi_{\ell}^{p-1} (u-\ul)\nabla\phi_{\ell} \}\ dx\\
       &=I_{1}+p\left(I_{2}-I_{3}\right),
\end{split}
\end{equation}
where,
\begin{align*}
    I_{1}&=\int_{\Sl}\{|\nabla \ul|^{p-2}\nabla \ul-|\nabla u|^{p-2}\nabla u\}\cdot \nabla \psi_{\ell}\ dx,\\
    I_{2}&=\int_{\Sl} \phi_{\ell}^{p-1}(u-\ul)|\nabla \ul|^{p-2}\nabla \ul \cdot \nabla \phi_{\ell}\ dx\\
    I_{3}&= \int_{\Sl} \phi_{\ell}^{p-1}(u-\ul)|\nabla u|^{p-2}\nabla u \cdot \nabla \phi_{\ell} \ dx.
\end{align*}
To manage $I_{1}$: As $\psi$ is a test a function, by Corollary \ref{corollary1} and Poincar\'e inequality, we get that
\begin{equation*}
    \begin{split}
        |I_{1}| &\leq \int_{\Sl}\phi_{\ell}^{p}|f(\ul)-f(u)||u-\ul|\ dx \\
        &\leq \int_{K}|\phi_{\ell}|^{p}\left(|f(\ul)|+|f(u)|\right)\left(|\ul|+|u|\right) \ dx\\
        &\leq 4 M(K)f(M(K))\int_{K}|\phi_{\ell}|^{p}\\
        &\leq  C_{1} \int_{K}|\nabla \phi_{\ell}|^{p}\\
        &\leq C_{1} \bigg(\frac{c_{1}}{\ell}\bigg)^{p},
    \end{split}
\end{equation*}
where $C_{1}=C_{1}(K, M(K),f)$. $M(K)$ also depends on $u$ by the Remark \ref{remark}.\\
To manage $I_{2}$ and $I_{3}$: Propositions \ref{prop_weak_local_bound} and \ref{prop_grad_lp_bound}
and H\"older inequality together implies
\begin{equation*}
    \begin{split}
        |I_{2}|&\leq \int_{\Sl} |\phi_{\ell}|^{p-1}|u-\ul||\nabla \ul|^{p-1}| \nabla \phi_{\ell}|\ dx\\
        &\leq 2M(K) \left(\frac{c_{1}}{\ell}\right)\int_{K} |\nabla \ul|^{p-1} \ dx\\
        &\leq 2M(K) \left(\frac{c_{1}}{\ell}\right) |K|^{\frac{1}{p}}\|\nabla \ul\|_{L^{p}(K)}^{p-1}\\
        & \leq  C_{2} \left(\frac{c_{1}}{\ell}\right),
    \end{split}
\end{equation*}
where $C_{2}=C_{2}(K, M(K))$. For the integral $I_{3}$, we take $a_{i}<b_{i}$; $i=1\cdots m$, such that $K\subset \displaystyle\prod_{1}^{m}[a_{i},b_{i}]\times\2$, then with similar steps as for $I_{2}$, we have
\begin{equation*}
    \begin{split}
        |I_{3}|&\leq 2M(K)\left(\frac{c_{1}}{\ell}\right)\int_{K}|\nabla u|^{p-1} \ dx\\
        & \leq 2M(K) \left(\frac{c_{1}}{\ell}\right) \prod_{1}^{m}|b_{i}-a_{i}|\int_{\2}|\nabla u|^{p-1} \ dX_{2}    \\
        &\leq C_{3} \left(\frac{c_{1}}{\ell}\right),
    \end{split}
\end{equation*}
where $C_{3}=C_{3}(K, \2, M(K))$. Thus, we have by the triangle inequality 
\begin{equation*}
\int_{\Tilde{S}}|\nabla(u-\ul)|^{p}\leq I\leq I_{1}+p(I_{2}+I_{3})\leq \frac{C}{\ell},
\end{equation*}
with $C=C(p,c_{1}, K, \2, M(K))$. The results follow.
 \end{proof}
 As the result is local in nature, the boundary values do not influence the argument.

 \section{Loss of uniqueness}
 The solution to \eqref{eqnonsl} and \eqref{bdryonsl_one} is unique regardless of the assumption \eqref{a2}; however, this is not true for large solutions. In this section, we will examine the case when the uniqueness of the large solution fails; the proceeding discussion is similar to \cite{IndroDatta2024}, and we only provide an outline of the proof. Unlike in the preceding section, where the choice of `$u$' was apparent, here, the proof of Theorem \ref{thrm2} requires four steps where the first is for defining $u$, see Proposition \ref{prop defining u}. Next is for showing that $u$ solves 
 \begin{equation}\label{eqn on S}
 \begin{cases}
     div\left(|\nabla u|^{p-2}\nabla u\right)=f(u)\quad &\text{in}\ S\\
      u(x)\xrightarrow[x\in S]{} \infty \quad &\text{as}\ dist(x,\partial S)\to 0,
 \end{cases}
\end{equation} see Theorem \ref{thrm u is the solution}. Third, finding relation between $u$ and the problem \eqref{eqnoncross}, see Theorem \ref{thrm u is the solution}. Finally, the proof from the previous section.
 \begin{proposition}\label{prop defining u}
     Assume \eqref{a1}. Let $\ul$ be solution to \eqref{eqnonsl} and \eqref{bdryonsl_two}, then there exists $u\in L_{loc}^{\infty}(S)$ such that $\ul\to u$ in $L^{q}_{loc}(S)$ for any $1\leq q< \infty$. Further, for any compact subset $K\subset S$, $\ul\rightharpoonup u$ in $W^{1,p}(K)$.
 \end{proposition}
\begin{proof}
    Let $\ell_{1}<\ell_{2}$, then both $u_{\ell_{1}}$ and $u_{\ell_{2}}$ solves \eqref{eqnonsl} on $S_{\ell_{1}}$ weakly. Taking boundary conditions into account, given the comparison principle, we have $u_{\ell_{1}}\geq u_{\ell_{2}}$ on $S_{\ell_{1}}$. That is $\{\ul(x)\}_{\ell\geq \ell(x)}$ is a decreasing sequence for all $x\in S$ and for some $\ell(x)>0$ depending on $x$. Define 
    \begin{equation}\label{dfn of u}
        u(x):=\displaystyle\lim_{\ell\to\infty} \ul(x) \quad\text{for any}\ x\in S.
    \end{equation}
    Then $\ul\to u$ in $L_{loc}^{q}(S)$ for any $q\in [1,\infty)$ by monotone convergence theorem. Next, by Proposition \ref{prop_weak_local_bound} and Proposition \ref{prop_grad_lp_bound} and since $1<p<\infty$, there exists $v\in W^{1,p}(K)$ such that $\ul\rightharpoonup v$, then $\ul\to v$ in $L^{p}(K)$, thus, $u=v$.
\end{proof}
 \begin{theorem}\label{thrm u is the solution}
 Assume \eqref{a1}
 \begin{enumerate}
     \item $u$ is a weak solution of \eqref{eqn on S}.
     \item  The solution $u$ is independent of $X_{2}$  variable. if we denote $v(X_{1})=u(X_{1},X_{2})$ for a fixed $X_{2}$, then $v$ solves \eqref{eqnoncross} and \eqref{bdryoncross_two}.
 \end{enumerate}
 \end{theorem}
\begin{proof}
    \textbf{Case 1: $2\leq p<\infty$} For this case we refer to \cite{IndroDatta2024}.\par

    \textbf{Case 2:$1<p<2$:} By the definition of local weak solution, choose $K\subset\subset S$. Let $K_{1}\subset\subset S$ be such that $K\subset\subset K_{1}$. For any $\ell>0$ choose $\phi_{\ell}\in C_{c}^{\infty}(K_{1})$ be such that $\phi_{\ell}=1$ on $K$ and $\nabla \phi_{\ell}\leq \frac{c}{\ell}$, for some positive constant $c$. Then, by H\"older inequality with exponent $2/p$ we can write
   \begin{equation*}
       \begin{split}
           \int_{K}|\nabla(\ul-u)|^{p}\ dx &\leq \int_{K_{1}}|\nabla(\ul-u)|^{p}\phi_{\ell}^{} \ dx\\
           &\leq \int_{K_{1}}\left(|\nabla \ul|+|\nabla u|\right)^{\frac{(p-2)p}{2}}|\nabla (\ul-u)|^{p}\phi_{\ell}^{} \left(|\nabla \ul|+|\nabla u|\right)^{\frac{(2-p)p}{2}}\ dx\\
           &\leq J_{1} J_{2}\\
       \end{split}
   \end{equation*}
   where,
   \begin{equation*}
       J_{1}=\left(\int_{K_{1}}\left(|\nabla \ul|+|\nabla u|\right)^{p-2}|\nabla (\ul-u)|^{2} \phi_{\ell}^{\frac{2}{p}}\ dx\right)^{\frac{p}{2}},
   \end{equation*}
   and
   \begin{equation*}
       J_{2}= \left(\int_{K_{1}}\left(|\nabla \ul|+|\nabla u|\right)^{p}\ dx \right)^{\frac{2-p}{2}}.
   \end{equation*}
   By Proposition \ref{prop_grad_lp_bound}, $J_{2}$ is bounded above by a positive constant independent of $\ell$.\\
   By the vector inequality 
    \begin{equation*}
        \left<|x|^{p-2}x-|y|^{p-2}y \cdot x-y\right>\geq C \left(|x|+|y|\right)^{p-2} |x-y|^{2}
    \end{equation*}
   for any $x,y\in \R^{n}$ and a constant $C>0$, we infer
   \begin{equation*}
       J_{1}\leq \int_{K_{1}}\left<|\nabla \ul|^{p-2}\nabla \ul- |\nabla u|^{p-2}\nabla u \cdot \nabla(\ul-u)\right>\phi_{\ell}^{\frac{2}{p}}.
   \end{equation*}
  Rest of the proof is similar to \cite[Theorem 4.3]{IndroDatta2024}:   By weak convergence 
   \begin{align*}
         \int_{K_{1}}\phi_{\ell}^{\frac{2}{p}} |\nabla u|^{p-2}\nabla u \cdot \nabla (u_{\ell}-u)\to 0\quad \text{as} \quad \ell\to\infty,
     \end{align*}
     and for the other term, consider \begin{align*}
        \int_{K_{1}}|\nabla u_{\ell}|^{p-2}\nabla u_{\ell} \cdot \phi_{\ell}^{\frac{2}{p}} \nabla (u_{\ell}-u) &=\int_{K_{1}}|\nabla u_{\ell}|^{p-2}\nabla u_{\ell} \cdot \{\nabla(\phi_{\ell}^{\frac{2}{p}}(u_\ell-u))-(u_{\ell}-u)\nabla \phi_{\ell}^{\frac{2}{p}}\}\\
        &=I_{1}-I_{2}.
    \end{align*} 
     By uniformly bound of $\|\nabla u_{\ell}\|_{L^{p}}$,  there is a $N>0$ such that
    \begin{align*}
        |I_{2}|\leq \int_{K_{1}}|\nabla u_{\ell}|^{p-1}|\nabla \phi_{\ell}^{\frac{2}{2}}\|u_{\ell}-u| &\leq  \sup_{K_{1}}|\nabla\phi_{\ell}^{\frac{2}{p}}| \|\nabla u_{\ell}\|_{L^{p}}^{p-1} \|u_{\ell}-u\|_{L^{p}}\\
        &\leq N \|u_{\ell}-u\|_{L^{p}} \xrightarrow{\ell \to 0}  0.
    \end{align*}
    Rewriting $I_{1}$,
    \begin{align*}
        I_{1}
        =\int_{K_{1}} \big(|\nabla u_{\ell}|^{p-2}\nabla u_{\ell} \cdot \nabla (\psi(u_{\ell}-u))- f(u)(u_{\ell}-u)\psi \big)  +  \int_{K_{1}} f(u)(u_{\ell}-u)\psi .
    \end{align*}
The second integral converges to $0$ by \ref{prop defining u}. For $\psi\in C_{c}^{\infty}(K_{1})$, by dominated convergence theorem we have
\begin{align*}
\int_{\Tilde{S}}f(u_{\ell})\psi\to \int_{\Tilde{S}}f(u)\psi.
   \end{align*}
    That is \begin{equation*}
        \int_{\Tilde{S}}|\nabla u_{\ell}|^{p-2}\nabla u_{\ell}\cdot \nabla\psi\to  \int_{\Tilde{S}}f(u)\psi.
    \end{equation*}
Then, $I_{1}\to 0$ by taking $\psi=(\ul-u)\phi$, for more detains we refer to \cite[Theorem 4.3]{IndroDatta2024}.
\end{proof}
The proof of the Theorem \ref{thrm2} is the same as the one in Section \ref{proofs of the theorems} with $\ui$ replaced by $u$.

\section*{Acknowledgements}
The author is supported by PMRF grant (2302262).\par
Special thanks to Dr. Indranil Chowdhury (Department of Mathematics and Statistics, Indian Institute of Technology Kanpur) for the support and suggestions.

\renewcommand\refname{Bibliography}
\bibliographystyle{abbrv}
\bibliography{ref.bib}
\end{document}